\documentclass[12pt,draft,leqno]{article}
\usepackage{amssymb, eucal, latexsym}

\usepackage{amsmath,amsthm,amssymb}
\usepackage{amsfonts}
\newtheorem{thm}{Theorem}[section]
\newtheorem{cor}[thm]{Corollary}

\newtheorem{prop}[thm]{Proposition}
 \newcommand{\Real}{\mathbb{R}}

 \newcommand{\spfi}{\mathbf{SpFi}}

\newcommand{\lspfi}{\mathbf{LSpFi}}
\newcommand{\comp}{\mathbf{Comp}}

\newcommand{\sets}{\mathbf{Sets}}
\newcommand{\w}{\mathbf{W}}
\newcommand{\F}{\mathcal{F}}
\newcommand{\XF}{(\mathbf{X},\mathcal{F})}
\newcommand{\YF}{\left(\mathbf{Y},\mathcal{F'}\right)}
\title{{\LARGE\bf
Monomorphisms in spaces with Lindel\"{o}f filters via some compact-open-like topologies on $C(X)$}\\
\vspace{0.35cm}
{\large\bf Vasil Gochev}\thanks{This paper was supported by the project No. 140/08.05.2014 ,,Function spaces and
dualities" of the Sofia University ,,St. Kl. Ohridski".
}\\
\vspace{0.25cm}
 {\footnotesize Dept. of Math. and
Informatics, Sofia University,  5 J. Bourchier Blvd., 1164 Sofia,
Bulgaria}
}

\author{}

\date{}

\begin{document}
\maketitle
\begin{abstract}
An object in the category $\spfi$ of spaces with filters, is a pair $\XF$, where $X$ is a compact Hausdorff space and $\F$ is a filter of dense open subsets of $X$. A morphism $f:\YF\rightarrow\XF$ is a continuous map $f:\mathbf{Y}\rightarrow \mathbf{X}$ for which $f^{-1}(F)\in\F'$ whenever $F\in\F$. In the present work we study the categorical monomorphisms in the subcategory $\lspfi$ of spaces with Lindel\"{o}f filters, meaning filters with a base of Lindel\"{o}f, or cozero sets. Of course, these monomorphisms need not be one-to-one. We extend the criterion derived by R. Ball and A. Hager in \cite{BH2}. We define and study some compact-open-like topologies on the set $C(X)$ of continuous real-valued functions defined on a topological space $X$. We use them to give a new characterization of the monomorphisms in $\lspfi$.
\end{abstract}

\footnotetext[1]{{\footnotesize
{\em Key words and phrases:} $C(X)$, \v{C}ech-Stone compactification, epimorphism, monomorphism, epi-topology,  compact-open topology, space with filter, frame, lattice-ordered group.}}

\footnotetext[2]{{\footnotesize
{\em 2010 Mathematics Subject Classification:} 54C35, 18A20, 54A05, 54A10, 54B05.}}

\footnotetext[3]{{\footnotesize {\em E-mail address:}
vgotchev@fmi.uni-sofia.bg}}

%\vspace{1.5cm}

\baselineskip = \normalbaselineskip

\section{Introduction}

%\centerline{{\bf 1. Introduction}}

\bigskip

The category $\spfi$ of spaces with filters arises naturally from
considerations in ordered algebra, e.g., Boolean algebra,
lattice-ordered groups and rings, and from considerations in general
topology, e.g., the theory of the absolute and other covers,
locales, and frames (see \cite{BHM}). An object in $\spfi$ is a pair $\XF$, where $X$
is a compact Hausdorff space and $\F$ is a filter of dense open
subsets of $X$. A morphism $f:\YF\rightarrow\XF$ is a continuous map
$f:\mathbf{Y}\rightarrow \mathbf{X}$ for which $f^{-1}(F)\in\F'$
whenever $F\in\F$. The subcategory $\lspfi$ of spaces with Lindel\"{o}f filters is defined as follows: $(X,\mathcal{F})\in \lspfi$ if and only if
$(X,\mathcal{F})\in \spfi$ and $\mathcal{F}$ has a base of cozero
sets.
%Since $X$ is compact, every cozero subset of $X$ is a countable union of compact subsets, so it is Lindel\"{o}f.

Main references for topology are \cite{EC}, \cite{E}, and \cite{GJ}. All
topological spaces are completely regular Hausdorff, usually
compact. $\comp$ is the category of compact (Hausdorff) spaces with
continuous maps. For a space $X$, $C(X)$ is the set (or
\emph{l}-group, vector lattice, ring, \emph{l}-ring,...) of
continuous real-valued functions on $X$. The \emph{cozero set} of
$f\in C(X)$ is $\mathrm{coz}f\equiv\{x|f(x)\neq0\}$, and
$\mathrm{coz}X\equiv \{\mathrm{coz}f|f\in C(X)\}$. Each cozero set in a compact space $X$
is an $F_{\sigma}$, hence Lindel\"{o}f.

In a general category, a monomorphism is a morphism $m$ which is
left--cance- lable: $mf=mg$ implies $f=g$. A one-to-one map is a
monomorphism in $\spfi$ since it is a monomorphism in the category
$\comp$ of compact Hausdorff spaces with continuous maps since it is
a monomorphism in $\sets$, but a monomorphism in $\spfi$ is not
necessarily one-to-one. Such examples one can find in \cite{BH2}.

This work focuses on the characterization of monomorphisms in the
category $\lspfi$ via compact-open-like topologies defined on
$C(X)$. More precisely, we derive a new necessary and sufficient condition
a morphism in $\lspfi$ to be monomorphism (see Theorem \ref{7.3}(4) below). Based on this condition,
for every object $\XF\in\spfi$ are defined
topologies $\tau^{\mathcal{F}}$ and
$\tau^{\mathcal{F}}_{\varepsilon=0}$ on $C(X)$ for which the following characterizations are derived
\begin{cor}\emph{(Corollary \ref{2.4.4} below)} $f:\XF\rightarrow\YF$ is a monomorphism in $\lspfi$ if and only if
$\widetilde{f}(C(Y))$ is dense in
$(C(X),\tau^{\mathcal{F}}_{\varepsilon=0})$.
\end{cor}
\begin{cor} \emph{(Corollary \ref{2.4.5} below)} $f:\XF\rightarrow\YF$ is a monomorphism in $\lspfi$ if and only if
$\widetilde{f}(C(Y))$ is dense in $(C(X),\tau^{\mathcal{F}})$.
\end{cor}
Here, the map $\widetilde{f}:C(Y)\rightarrow C(X)$ is defined by
$\widetilde{f}(g)= g\circ f$ for every $g\in C(Y)$.

We have shown in \cite{G1} that these topologies are always $T_{1}$, countably tight, and homogeneous. We established necessary and sufficient conditions these topologies to be Hausdorff or group topologies. We don't include these results here because of the significant extra complications and length.
There are different characterizations of the monomorphisms in $\spfi$
and $\lspfi$ which are briefly summarized in \cite{BH2}, but for
more details one can see also \cite{BHM}, \cite{BHMa}, \cite{H}, \cite{H1}, \cite{JM}, and \cite{MM}.
%%% ----------------------------------------------------------------------

\section{Monomorphisms in $\lspfi$}  First
we will recall some of the results established by R. Ball
and A. Hager in \cite{BH2}.
\begin{thm}\label{7.1}
\emph{(\cite{BH2})} Let $\YF \overset{f}{\leftarrow}\XF \in\lspfi$.
Then the following are equivalent.

\emph{(1)} $f$ is a monomorphism in $\lspfi$.

\emph{(2)} If $x_{1}\neq x_{2} \in X $ then there are $S\in
\mathcal{F}_{\delta}$ and neighborhoods $U_{i}$ of $x_{i}$ for which
$$f[U_{1}\cap S]\cap f[U_{2}\cap S]=\emptyset .$$

\emph{(3)} If $K_{1}$ and $K_{2}$ are disjoint compact sets in $X$
then there are $S\in \mathcal{F}_{\delta}$ and neighborhoods $U_{i}$
of $K_{i}$ for which $$f[U_{1}\cap S]\cap f[U_{2}\cap S]=\emptyset
.$$

\emph{(4)} For each $b\in C(X)$ there is $S\in \mathcal{F}_{\delta}$
for which $x_{1}, x_{2} \in S $ and $f(x_{1})=f(x_{2})$ imply
$b(x_{1})=b(x_{2})$.
\end{thm}
\begin{cor}\emph{(\cite{BH2})} If there is $S\in
\mathcal{F}$ or $S\in \mathcal{F}_{\delta}$ such that $f$ is
one-to-one on $S$ then $f$ is monic.
\end{cor}

\begin{proof} {
If there is $S\in
\mathcal{F}$ or $S\in \mathcal{F}_{\delta}$ such that $f$ is
one-to-one on $S$ then condition (2) of the previous theorem holds,
so $f$ is a monomorphism.}
\end{proof}

In the following theorem, the equivalence of (1), (2), and (3) is from
\cite{BH2}. We have included condition (4).

\begin{thm} \label{7.3}Let $\YF
\overset{f}{\leftarrow}\XF \in\lspfi$, and let
$A\equiv\widetilde{f}(C(Y))$. Then the following are equivalent.

\emph{(1)} $f$ is a monomorphism in $\lspfi$.

\emph{(2)} For every $b\in C(X)$ there is an
$E\in\mathcal{F}_{\delta}$ such that for every $x_{1},x_{2}\in E$
there is $a\in A$ with $ax_{1}=bx_{1}$ and $ax_{2}=bx_{2}$.

\emph{(3)} For every $b\in C(X)$ there is an $E\in
\mathcal{F}_{\delta}$ such that for every finite $F\subseteq E$
there is $a\in A$ with $a|F=b|F$.

\emph{(4)} For every $b\in C(X)$ there is an $E\in
\mathcal{F}_{\delta}$ such that for every compact $K\subseteq E$
there is $a\in A$ with $a|K=b|K$.

\end{thm}
\begin{proof} The equivalence of the first three conditions in the
theorem is proven in \cite{BH2}. The implication $(4)\Rightarrow
(3)$ is clear. We will show that the condition $(4)$ of Theorem
\ref{7.1} implies (4) of this theorem. Let us fix $b$ and $S\in
\mathcal{F}_{\delta}$ as in Theorem \ref{7.1}$(4)$. Then for every
compact $\emptyset\neq K\subseteq S$, $f(K)$ is a compact subset of
$Y$. Let us define on $f(K)$ a function $g:f(K)\rightarrow\Real$ as
follows $g(y)=b(x)$, where $x\in K$ and $y=f(x)$. The function $g$
is well defined, because on $f^{-1}(y)\cap K$ the function $b$ is
equal to $b(x)$. If $g$ is continuous, then we can extend it to
$\widetilde{g}:Y\rightarrow\Real$ by the normality of $Y$. Then $a=
\widetilde{g}\circ f$ has the  required property $a|K=b|K$. Thus we
need to show that $g$ is continuous function.

Let us consider the following general situation: $f:X\rightarrow Y$
is a continuous onto mapping where $X$ and $Y$ are compact Hausdorff
spaces. Let $X_{/E(f)}$ be the quotient space under the equivalence
relation $E(f)$ induced by $f$. It is clear that $f$ is closed
mapping. Now from Corollary 2.4.8 from \cite{E} it follows that $f$
is a quotient mapping and by Proposition 2.4.3 from \cite{E} it
follows that the mapping $\overline{f}:X_{/E(f)}\rightarrow Y$ is a
homeomorphism. We have the following situation: $X,Y,Z$ -
topological spaces, $X,Y$ - compact Hausdorff, $f:X\rightarrow Y$ is
a continuous onto mapping, $g:X\rightarrow Z$ is continuous and has
the property $g_{/f^{-1}(y)}= constant=z(y)\in Z$ for every $y\in
Y$. The map $\widetilde{g}: Y\rightarrow Z$ is defined by
$\widetilde{g}(y)=z(y)$, $\overline{f}:X_{/E(f)}\rightarrow Y$ is
defined by $\overline{f}(f^{-1}(y))=y$, and $q:X\rightarrow
X_{/E(f)}$ is the natural quotient mapping. From the previous notes
it follows that $\overline{f}:X_{/E(f)}\rightarrow Y$ is a
homeomorphism, so $\widetilde{g}$ is continuous if and only if
$\widetilde{g}\circ \overline{f}$ is continuous if and only if (by proposition
2.4.2 from \cite{E}) $\widetilde{g}\circ \overline{f}\circ q =
\widetilde{g}\circ f = g $ is continuous. This implies that the
function $g:f(K)\rightarrow\Real$ defined above is continuous.
\end{proof}
Condition (4) from the previous theorem suggests on $C(X)$ a topology $\tau$ such that $A$ is a dense subset in $(C(X),\tau)$. In the following section we define two topologies on $C(X)$ which have that property and characterize the monomorphisms in $\lspfi$.

\section{Some compact-open-like topologies on $C(X)$}
Let $X$ be a topological space and $S$ be a dense subset of $X$. Let
$\tau_{S}$ be the relativization of the compact-open topology
$\tau_{co}$ on $C(S)$ to $C(X)$. (More details on the definition end properties of the compact-open topology one can find in \cite{E} and \cite{MN}.) Here we consider $C(X)$ as a
subspace of $(C(S),\tau_{co})$ via the natural embedding $C(X)\ni
f\mapsto f|S \in C(S)$. Since $(C(S),\tau_{co},+,\vee,\wedge)$ is a
Hausdorff topological \emph{l}-group with $+,\vee,\wedge$ defined pointwise, so is $(C(X),\tau_{S},+,\vee,\wedge)$ and the neighborhoods of the constant function $0$ suffice to describe it.

If $\mathcal{S}$ is a family of dense subsets of a topological space
$X$, then for each $S\in \mathcal{S}$ we have the corresponding
$\tau_{s}$ as above, and we consider the $\mathcal{S}$-topology
$\tau^{\mathcal{S}}\equiv\wedge\{\tau_{s}|S\in \mathcal{S}\}$ on $C(X)$. Here $\wedge$ is taken in the
lattice of topologies on $C(X)$. It is now unclear what are the properties of the structures $(C(X),\tau^{\mathcal{S}})$? Topology $\tau^{\mathcal{S}}$ is always $T_{1}$, but not necessarily Hausdorff or group topology. (See \cite{HR} for the definition of topological group and more.)

The above is surely (or may be not) too general to say much. We
limit the generality to the following considerations: Let $(X,\mathcal{F})\in \lspfi$ and $\mathcal{F}_{\delta}\equiv\{\cap F_{n}|F_{1},F_{2},F_{3}\ldots\in\mathcal{F}\}$. For $S\in
\mathcal{F}_{\delta}$, let $\tau_{S}$ be as above and $\mathcal{K}(S)$ is the family of all compact subsets of $S$. We define $\tau^{\mathcal{F}}\equiv\wedge\{\tau_{s}|S\in \mathcal{F}_{\delta}\}$. A local base at a point $f\in
C(X)$ is the collection of all sets of the form
$\underset{S\in\mathcal{F}_{\delta}}{\bigcup}U^{s}(K_{s},\varepsilon_{s},f)$,
where $K_{s}\in\mathcal{K}(S)$ and $\varepsilon_{s}\in(0,1]$ for
every $S\in\mathcal{F}_{\delta}$.

The following proposition follows immediately from definitions.
\begin{prop}\label{SF2} Let $f:\XF\rightarrow\YF$ be a morphism in $\lspfi$,
$\tau^{\mathcal{F}}$ and $\tau^{\mathcal{F'}}$ be the corresponding
topologies on $C(X)$ and $C(Y)$. Let $\widetilde{f}:C(Y)\rightarrow
C(X)$ be the induced map defined by $\widetilde{f}(g)= g\circ f$ for
every $g\in C(Y)$. Then
$\widetilde{f}:(C(Y),\tau^{\mathcal{F'}})\rightarrow
(C(X),\tau^{\mathcal{F}})$ is continuous.
\end{prop}

Now we are going to define a topology on the set of continuous
functions which is suggested by condition (4) from Theorem \ref{7.3}. Let $X$ be a Tychonoff space, and let $\mathcal{K}(X)$ be the family of all compact subsets of $X$. We define the \emph{compact-zero}
topology $\tau_{\varepsilon=0}$ on $C(X)$ as follows: Basic neighborhoods for $f\in C(X)$ are of the form
$U(K,f)=\{g\in C(X)|f/K=g/K\}$, where $K\in\mathcal{K}(X)$. It is straightforward to check that $\{U(K,f)|f\in C(X),K\in\mathcal{K}(X)\}$ form a base for a topology on $C(X)$. This topology coincides with the discrete topology if and only it $X$ is compact. It is
clear that $U(K,f)=\{g\in C(X)|x\in K \Rightarrow |f(x)-g(x)|=0\}$,
i.e. we allow $\varepsilon=0$ in the usual compact-open topology.
Obviously this topology is finer than the compact-open topology on
$C(X)$. One can easily check that $+,-,\vee,\wedge$ are continuous, thus $(C(X),\tau_{\varepsilon=0},+,\vee,\wedge)$ is a Hausdorff topological \emph{l}-group. Let $(X,\mathcal{F})\in \lspfi$ and for a fixed $S\in
\mathcal{F}_{\delta}$ let $\tau^{S}_{\varepsilon=0}$ be the relativization of the compact-zero topology on $C(S)$ to $C(X)$. We define the topology
$\tau^{\mathcal{F}}_{\varepsilon=0}\equiv\wedge\{\tau^{S}_{\varepsilon=0}|S\in \mathcal{F}_{\delta}\}$. A local base at a point $f\in C(X)$ is the collection of all sets of the form $\underset{S\in\mathcal{F}_{\delta}}{\bigcup}U^{S}_{\varepsilon=0}(K_{s},f)$,
where $K_{s}\in\mathcal{K}(S)$ for every $S\in\mathcal{F}_{\delta}$.

The following proposition follows immediately from the definitions
\begin{prop}\label{SF22} Let $f:\XF\rightarrow\YF$ be a morphism in $\lspfi$,
$\tau^{\mathcal{F}}_{\varepsilon=0}$ and
$\tau^{\mathcal{F'}}_{\varepsilon=0}$ be the corresponding
topologies on $C(X)$ and $C(Y)$. Let $\widetilde{f}:C(Y)\rightarrow
C(X)$ be the induced map defined by $\widetilde{f}(g)= g\circ f$ for
every $g\in C(Y)$. Then
$\widetilde{f}:(C(Y),\tau^{\mathcal{F'}}_{\varepsilon=0})\rightarrow
(C(X),\tau^{\mathcal{F}}_{\varepsilon=0})$ is continuous.
\end{prop}
Propositions \ref{SF2} and \ref{SF22} show that the associations
$\XF\rightarrow(C(X),\tau^{\mathcal{F}}(\tau^{\mathcal{F}}_{\varepsilon=0}))$ are functorial(\cite{HS}).

The space $(C(X),\tau^{\mathcal{F}}_{\varepsilon=0}(\tau^{\mathcal{F}}))$ is always countably tight, homogeneous, and $T_{1}$, but not necessarily Hausdorff or topological group. Necessary and sufficient conditions for $\tau^{\mathcal{F}}_{\varepsilon=0}$ or $\tau^{\mathcal{F}}$ to be Hausdorff, or group topology on $(C(X),+)$ are derived in \cite{G1}. These considerations are not included here, because of their significant length and will be examined in a later paper. Of course, one can ask many questions about the properties of these topologies. For an inspiration of such questions one could see \cite{A} and \cite{MN}.

\section{Characterizations of monomorphisms in $\lspfi$ \\via $\tau^{\mathcal{F}}_{\varepsilon=0}$ and $\tau^{\mathcal{F}}$}
The main results in this work are the following corollaries.
\begin{cor} \label{2.4.4} Let $\YF
\overset{f}{\leftarrow}\XF \in\lspfi$, and let
$A\equiv\widetilde{f}(C(Y))$. Then $f$ is a monomorphism in $\lspfi$
if and only if $A$ is dense in
$(C(X),\tau^{\mathcal{F}}_{\varepsilon=0})$.
\end{cor}
\begin{proof} The proof is straightforward and follows from Theorem \ref{7.3}(4) and the fact that basic neighborhoods of a function  $f\in C(X)$ are of the form $\underset{S\in\mathcal{F}_{\delta}}{\bigcup}U^{S}_{\varepsilon=0}(K_{s},f)$,
where $K_{s}\in\mathcal{K}(S)$ for every $S\in\mathcal{F}_{\delta}$.
\end{proof}
\begin{cor}\label{2.4.5} Let $\YF \overset{f}{\leftarrow}\XF \in\lspfi$ and let $A\equiv\widetilde{f}(C(Y))$. Then $f$ is a monomorphism in $\lspfi$ if and only if $A$ is dense in
$(C(X),\tau^{\mathcal{F}})$.
\end{cor}
\begin{proof}
$\Rightarrow$ ) Obviously follows from Theorem \ref{7.3}(4).

$\Leftarrow$ ) For the converse we need some additional preliminary definitions and facts.

Let $\XF\in \lspfi$ and $X^{\ast}=\underset{\leftarrow}{\lim}\{\beta
F\,|\,F\in \mathcal{F}\}$, where $\beta F$ is the \v{C}ech-Stone
compactification of $F$. For $X^{\ast}$ the bonding maps are: when $F, F_{1}\in \mathcal{F}$ and $F\subseteq F_{1}$,  $\beta F_{1} \overset{\pi^{F}_{F_{1}}}{\leftarrow} \beta F$ is the
\v{C}ech - Stone extension of the inclusion. For  $F\in \mathcal{F}$, there is the projection $\beta F \overset{\pi_{F}}{\leftarrow} X^{\ast}$, thus there exists a continuous map
$\pi:X^{\ast}\rightarrow X$. This $\pi$ is irreducible, so inversely preserves dense sets ($\pi^{-1}[F]$ is dense in
$X^{\ast}$ for every $F\in\mathcal{F}$). Let $\mathcal{F}^{\ast}$ has
base $\{\pi^{-1}[F]\,|\,F\in\mathcal{F}\}$. Then
$\pi:(X^{\ast},\mathcal{F}^{\ast})\twoheadrightarrow \XF$ is a
morphism in $\lspfi$ and it is monic. (See \cite{BH2}.)
\begin{prop}\label{A} Let $\YF\overset{f}{\leftarrow}\XF \in\lspfi$ and $\pi$
be as above. Then $f$ is monic if and only if $f\circ\pi$ is monic.
\end{prop}
\begin{proof} If $f$ is monic then $f\circ\pi$ is monic as a composition
of two monomorphisms. For the converse we will use Theorem
\ref{7.3}(4). We have
\[\YF\overset{f}{\leftarrow}\XF\overset{\pi}{\leftarrow}
(X^{\ast},\mathcal{F}^{\ast}),\] and let
$A\equiv\widetilde{f}(C(Y))$ and $B\equiv\widetilde{\pi}(C(X))$. Let
$b\in C(X)$ then $\widetilde{\pi}(b)\in C(X^{\ast})$ and since
$f\circ\pi$ is monic, there exists $D\in\mathcal{F}^{\ast}_{\delta}$
such that for every compact $K\subseteq D$ there is $c\in
\widetilde{\pi}(A)=\widetilde{\pi}(\widetilde{f}(C(Y)))$ with
$c|K=\widetilde{\pi}(b)|K$. Since
$\{\pi^{-1}[F]\,|\,F\in\mathcal{F}\}$ is a base for
$\mathcal{F}^{\ast}$, we can take $D=\pi^{-1}[E]$ where $E\in
\mathcal{F}_{\delta}$. Let $K\in \mathcal{K}(E)$, then
$\pi^{-1}[K]\in \mathcal{K}(D)$, so there is a map $c\in
\widetilde{\pi}(A)$  such that $c|\pi^{-1}[K]
=\widetilde{\pi}(b)|\pi^{-1}[K]$. Since $c\in\widetilde{\pi}(A)$,
there is a map $a\in A$ such that $c=a\circ\pi$. Therefore
$a\circ\pi|\pi^{-1}[K] =b\circ\pi|\pi^{-1}[K]$, but that means
$a|K=b|K$.
\end{proof}
Let $SY$ be the ``$\spfi$ Yosida functor" $\w
\overset{SY}{\rightarrow}\lspfi$ (\cite{BHM} and \cite{Y}). Here $\w$ is the category of archimedean $l$-groups with distinguished weak order unit, with $l$-group homomorphisms which preserve unit. This category includes all
rings of continuous functions $C(X)$. Consider
$\XF\overset{\pi}{\twoheadleftarrow}
(X^{\ast},\mathcal{F}^{\ast})\in\lspfi$ and let
$C[\mathcal{F}]=\underset{\rightarrow}{\lim}\{C(F)\,|\,F\in
\mathcal{F}\}\in\w$. We have $\widetilde{\pi}:C(X)\rightarrow
C[\mathcal{F}]\in \w$, defined as $\widetilde{\pi}(f)=f\circ\pi$.
Let $B\equiv\widetilde{\pi}(C(X))\leq C[\mathcal{F}]$. Since $\pi$
is monic,
$(X^{\ast},\mathcal{F}^{\ast})\overset{\pi'}{\twoheadrightarrow}(X,\{X\})$
is monic too. $\pi'=SY(\widetilde{\pi})$, so $\widetilde{\pi}$ is
$\w$-epic (\cite{BH1}). Therefore $B$ is
$\tau^{\mathcal{F}^{\ast}}$-dense in $C[\mathcal{F}]$ (\cite{BH1}). Thus $B$ is $\tau^{\mathcal{F}^{\ast}}$-dense in
$C(X^{\ast})$.
\begin{prop}\label{AA} Consider
$(X^{\ast},\mathcal{F}^{\ast})\overset{\pi'}{\twoheadrightarrow}(X,\{X\})$.
Then $$\widetilde{\pi}:(C(X),\tau^{\mathcal{F}})\rightarrow
(C(X^{\ast}),\tau^{\mathcal{F}^{\ast}})$$ is a topological and
algebraic embedding.
\end{prop}
\begin{proof}The proof is straightforward and follows from the following facts: \\ (1) $\pi$ has the property: $ \forall\,E\in
\mathcal{F}^{\ast}_{\delta}\,\exists \,S_{E}\in
\mathcal{F}_{\delta}\,(\pi^{-1}[S_{E}]\subseteq E).$ \\ (2)
$\widetilde{\pi}:(C(X),\tau^{S_{E}})\rightarrow
(C(X^{\ast}),\tau^{E})$ is a topological and algebraic embedding.
(See also \cite{EC} and \cite{MN}.)
\end{proof}
\begin{cor} Let $\YF\overset{f}{\twoheadleftarrow}\XF \in\lspfi$ and
$A\equiv\widetilde{f}(C(Y))$ be dense in
$(C(X),\tau^{\mathcal{F}})$. Then $f$ is monic.
\end{cor}
\begin{proof} By Proposition \ref{A} it suffices to show that
$\YF\overset{f}{\leftarrow}\XF\overset{\pi}{\leftarrow}
(X^{\ast},\mathcal{F}^{\ast})$ is monic.
$B\equiv\widetilde{\pi}(C(X))$ is $\tau^{\mathcal{F}^{\ast}}$-dense
in $C[\mathcal{F}]$ and by Proposition \ref{AA}
$(C(X),\tau^{\mathcal{F}})$ is homeomorphic to $B$. Therefore $A$ is
$\tau^{\mathcal{F}^{\ast}}$-dense in $C[\mathcal{F}]$. Thus
$A\overset{\lambda}{\leq}C[\mathcal{F}]$ is $\w$-epic. Then
$SY(\lambda)$ is $\lspfi$-monic. But $SY(\lambda)=f\circ\pi$.
\end{proof}
This completes the proof of Corollary \ref{2.4.5}.
\end{proof}

Let us note, that the condition (3) from Theorem \ref{7.3} suggests another topology
on $C(X)$. More precisely, for any $S\in \mathcal{F}_{\delta}$ one could consider the
topology of pointwise convergence (see \cite{A}) on $C(S)$ and the relativization of this topology to $C(X)$.
If it is denoted by $\sigma_{s}$, then define the topology
$\sigma^{\mathcal{F}}\equiv\wedge\{\sigma_{S}|S\in \mathcal{F}_{\delta}\}$.
This topology is coarser than the topology $\tau^{\mathcal{F}}$. For a topological space $Y$, let $Fin(Y)$ be the family of all finite subsets of $Y$. For a function $f\in C(X)$ the basic neighborhoods are of the form $\underset{S\in\mathcal{F}_{\delta}}{\bigcup}U^{S}(K_{S},\varepsilon_{S},f)$,
where $K_{S}\in Fin(S)$ and $\varepsilon_{S}\in(0,1]$ for
every $S\in\mathcal{F}_{\delta}$. Sets of the form $U^{S}(K_{S},\varepsilon_{s},f)\equiv\{g\in C(X)|x\in K_{S}\Rightarrow |f(x)-x(x)|<\varepsilon\}$, where $K_{S}\in Fin(S)$ and $\varepsilon_{S}\in(0,1]$, are the basic neighborhoods of $f$ in $(C(X), \sigma_{S})$. From this fact and Theorem \ref{7.3}(3) the following corollary is immediate.
\begin{cor} \label {12} Let $\YF\overset{f}{\twoheadleftarrow}\XF \in\lspfi$ be a monomorphism in
$\lspfi$ and $A\equiv\widetilde{f}(C(Y))$. Then $A$ is dense in
$(C(X),\sigma^{\mathcal{F}})$.
\end{cor}

In a similar way one can define the topology $\sigma^{\mathcal{F}}_{\varepsilon=0}$ as
the meet of the topologies $\{\sigma^{S}_{\varepsilon=0}|S\in \mathcal{F}_{\delta}\}$. Here in the definition of the topology of pointwise convergence we again alow $\varepsilon=0$, i.e. sets of the form $U^{S}(K_{S},f)\equiv\{g\in C(X)|x\in K_{S}\Rightarrow |f(x)-x(x)|=0\}$, where $K_{S}\in Fin(S)$, are the basic neighborhoods of $f$ in $(C(X), \sigma^{S}_{\varepsilon=0})$. Thus, for a function $f\in C(X)$ the basic neighborhoods are of the form $\underset{S\in\mathcal{F}_{\delta}}{\bigcup}U^{S}(K_{S},f)$,
where $K_{S}\in Fin(S)$ for every $S\in\mathcal{F}_{\delta}$.
Now, it is clear that the following corollary is true.

\begin{cor} Let $\YF\overset{f}{\twoheadleftarrow}\XF \in\lspfi$ and $A\equiv\widetilde{f}(C(Y))$. Then $f$ is a monomorphism in $\lspfi$ if and only if $A$ is dense in
$(C(X),\sigma^{\mathcal{F}}_{\varepsilon=0})$.
\end{cor}

Let us mention, that the authors of \cite{BH2} said some words about a possible topology on $C(X)$, but left this problem open. We do not know is the converse of Corollary \ref{12} also true?

%\vskip0.2in The author would like to thank his advisor Dr. A. Hager for the careful reading of this work and for the many and useful suggestions during its preparation.

\end{document}